\documentclass{amsart} 
\usepackage{enumerate,hyperref,amsthm}
\newcommand{\set}[1]{\left\{#1\right\}}
\newcommand{\norm}[1]{\left\lVert#1\right\rVert}

\newcommand{\pair}[1]{\left\langle#1\right\rangle_H}
\newcommand{\ind}[1]{1\hspace{-.28em}\mathrm{I}_{#1}}
\newcommand{\abs}[1]{\left\vert#1\right\vert}
\newcommand{\aabs}[1]{\big\lvert#1\big\rvert}
\newcommand{\eps}{\varepsilon}
\newcommand{\db}{\Delta B^H}
\newcommand{\dw}{\Delta W}
\newcommand{\al}{\alpha}
\newtheorem{theorem}{Theorem}[section]
\newtheorem{lemma}{Lemma}[section]

\theoremstyle{remark}
\newtheorem{remark}{Remark}[section]
\theoremstyle{definition}

\newcommand{\R}{\mathbb{R}}
\newcommand{\xd}{X^\delta}
\newcommand{\xdc}[1]{X^\delta_{t_{#1}^\delta}}
\newcommand{\td}{t^\delta}

\newcommand{\F}{\mathcal{F}}
\newcommand{\ex}[1]{\mathsf{E}\left[#1\right]}

\title[Rate of convergence of Euler  approximations of mixed SDE]{Rate of convergence of Euler  approximations of solution to mixed stochastic differential equation involving Brownian
motion  and fractional Brownian motion}
\author{Yuliya S. Mishura}

\address{\it Department of Probability, Statistics
and Actuarial Mathematics, Mechanics and Mathematics Faculty, Kyiv Taras Shevchenko National University,
Volodymyrska, 60, 01601 Kyiv, Ukraine}

\email{myus@univ.kiev.ua}
\author{Georgiy M. Shevchenko}
\address{\it Department of Probability, Statistics
and Actuarial Mathematics, Mechanics and Mathematics Faculty, Kyiv Taras Shevchenko National University,
Volodymyrska, 60, 01601 Kyiv, Ukraine}
\email{zhora@univ.kiev.ua}

\subjclass[2010]{60G15;
60G22; 60H10; 60J65}

\keywords{Fractional Brownian motion, mixed stochastic differential equation, pathwise integral, Euler approximation}

\begin{document}
\begin{abstract}
We consider a mixed stochastic differential equation
involving both standard Brownian motion and fractional Brownian
motion with Hurst parameter $H>1/2$. The mean-square rate of convergence of Euler approximations of
solution to this equation is obtained.
\end{abstract}

\maketitle%

\section*{Introduction}
The main object of this paper is the following mixed stochastic differential equation
involving independent  Wiener process $B$ and fractional Brownian motion $B^H$ with Hurst
index $H \in (1/2,1)$:
 \begin{equation}\label{main} X_t =X_0 +\int_0^t a(s,X_{s})ds+
\int_0^tb(s,X_{s})dW_s+\int_0^tc(X_{s})dB_s^H,\
t\in[0,T],\end{equation}
where the integral w.r.t.\ Wiener process is the standard It\^o integral, and the integral w.r.t.\ fBm is the forward
stochastic integral.
The questions of existence and uniqueness of solution for equations of such type were considered in  \cite{kub2,MisP,GN,MS}.

Such mixed equations arise in different applied areas. In financial mathematics, for example, it is often natural to assume that the underlying
random noise consists of two parts: a ``fundamental'' part, describing the economical background for a stock price, and a ``trading''
part, coming from the randomness inherent for the stock market. In this case the fundamental part of the noise should have a long memory,
while the second part is likely to be a white noise.

Due to a wide area of applications of equation \eqref{main}, it is important to consider certain numerical methods to solve it. We
use here the most popular and probably the simplest method of Euler approximations: one takes a uniform partition of the
interval, where the equation is being solved, and replaces differentials by a correspondent finite differences. There is a
vast literature dedicated to numerical methods for  stochastic differential equations driven by the Wiener process,
we refer to classical monographs \cite{milsh} and \cite{k-p} for an overview of the subject. There are also several papers
dealing with discrete time approximations for stochastic differential equations with fractional Brownian motion, for example, \cite{MSa,nourdin,davie}.

The main difficulty when considering equation \eqref{main} lies in the fact that the machinery behind the two stochastic
integrals is very different. The It\^o integral is treated usually in a mean square sense, while the integral with respect
to fractional Brownian motion is understood and controlled in a  pathwise sense. The mixture of two integrals makes things a lot harder,
forcing us to consider very smooth coefficients and to make delicate estimates.

The paper is organized as follows. In Section 1, we give basic fact about forward and Skorokhod
integration with respect to fractional Brownian motion and formulate main hypotheses. In Section 2, we define Euler approximations of \eqref{main} and establish some uniform integrability results for them. Section 3 contains the main result about rate of convergence of Euler approximations for equation \eqref{main}. Unsurprisingly, the rate of convergence appears to be equal to the worst of the rates for corresponding ``pure'' equations,
i.e.\ the mean-square distance between true and approximate solutions is of order $\delta^{1/2}\vee \delta^{2H-1}$, where $\delta$ is
the mesh of the partition.

\section{Preliminaries}

\subsection{Fractional Brownian motion and stochastic integration}
In this section we give basic facts about the stochastic calculus for fractional Brownian motion.
A more extensive exposition can be found e.g.\ in \cite{decre98,biaginietal}.

Fractional Brownian motion (fBm) $B^H$ is by definition a centered Gaussian process with the covariance
$$
\ex{B^H_t B_s^H} = \frac12 \left(t^{2H}+s^{2H} - \abs{t-s}^{2H}\right),\ t,s\ge 0.
$$
It has a version with almost surely $\kappa$-H\"older continuous paths for any $\kappa<H$. For
$H\in (1/2,1)$ (the case we consider here) it exhibits a property of long-range dependence.

Let $L^H_2[0,T]$ be the completion of the space of continuous functions with respect to the scalar
product
$$
\pair{f,g} = \iint_{[0,T]^2} f(t)g(s)\psi(t,s) ds\,dt,
$$
where $\psi(t,s) = H(2H-1)\abs{t-s}^{2H-2}$.
Denote also by $\norm{f}_H = \sqrt{\pair{f,f}}$ the corresponding norm.

Now we recall the notion of stochastic derivative. Let infinitely differentiable function $F\colon \R^n \to \R$ be
bounded along with derivatives. For a smooth functional $G=F(B^H_{t_1},\dots,B^H_{t_n})$, where $t_1,\dots,t_n\in[0,T]$
the stochastic derivative is defined as
$$
D_s G = \sum_{k=1}^n F'_{x_k} (B^H_{t_1},\dots,B^H_{t_n})\ind{[0,t_k]}(s).
$$
The Sobolev space $\mathbb D^{1,2}$ is the closure of the space of smooth functionals with respect to the norm
$$
\norm{G}_{1,2}^2 = \ex{G^2} + \ex{\norm {D\, G}_H^2}.
$$

The Skorokhod, or divergence, stochastic integral is the adjoint to the stochastic derivative in the following sense.
Let the domain $\mathrm{dom}\, \delta$ of the divergence integral be the space of random processes $u\in L_2(\Omega, L_2^H[0,T])$ such that
$$
\ex{\pair{D\,G, u}}\le C_u \norm{G}_{L_2(\Omega)}
$$
for all $G\in \mathbb D^{1,2}$. Then the divergence integral
$$
\delta(u) = \int_0^T u_t \delta B^H_t
$$
is defined as the unique element of $L_2(\Omega)$ such that
\begin{equation}\label{iBp}
\ex{\pair{D\,G, u}} = \ex{G\delta(u)}
\end{equation}
for all $G\in \mathbb D^{1,2}$. It is worth to remark that $\mathrm{dom}\,\delta$
contains the space $\mathbb D^{1,2}(L^2[0,T])$ of processes such that
$$
\norm{u}^2_{H;1,2} = \ex{\norm{u}_H^2} + \iiiint_{[0,T]^4}\ex{D_t u_v D_s u_z}\psi(t,s)\psi(v,z)ds\,dt\,dz\,dv
$$
is finite. Moreover, for such processes
\begin{equation}\label{skorintestim}
\ex{\delta(u)^2}\le \norm{u}^2_{H;1,2}.
\end{equation}

The forward integral with respect to fBm is defined as the uniform limit in probability
$$
\int_0^t u_s dB^H_s = \lim_{\eps\to 0} \int_0^t u_s \frac{B^H_{s+\eps} - B^H_s}{\eps} ds,
$$
provided this limit exists. It is well-known (see e.g.\ \cite{biaginietal}) that
if for $u\in \mathbb D^{1,2}(L_2^H[0,T])$
$$
\iint_{[0,T]^2}  \abs{D_s u_t} \psi(t,s)ds\,dt<\infty,
$$
then the forward integral exists and is equal to
\begin{equation}\label{forw-skor-rel}
\int_0^T u_t dB^H_t = \int_0^T u_t\delta B^H_t + \iint_{[0,T]^2} D_s u_t \psi(t,s)ds\,dt.
\end{equation}

\subsection{Assumptions}
The following hypotheses on the ingredients of equation \eqref{main} will be assumed throughout the paper.
\begin{list}{Hypotheses}{\parsep=-1mm}
\item[(A)] The functions $a$ and $b$ are bounded together with their derivatives $a'_x$, $b'_x$:
$$
\abs{a(t,x)}+ \abs{b(t,x)}+ \abs{a'_x(t,x)}+ \abs{b'_x(t,x)} \le K;
$$

\item[(B)] the functions $a$ and $b$ are uniformly $(2H-1)$-H\"older continuous in time:
$$
\abs{a(t,x)-a(s,x)}+ \abs{b(t,x)-b(s,x)}\le K\abs{t-s}^{2H-1};
$$

\item[(C)] the coefficient $c$ is bounded together with its first and second derivatives and uniformly
positive:
$$
0\le c(x)+c(x)^{-1}+ \abs{c'(x)}+ \abs{c''(x)}\le K.
$$
Here $K$ is  a constant independent of $x$, $t$ and $s$;

\item[(D)] the Wiener process $W$ and the fractional Brownian motion $B^H$ are independent.
\end{list}

In what follows $C$ will denote a generic constant, whose value might change from line to line. To emphasize
dependence on some variables, we will put them into subscript. For a random process $X$ we denote
its increments by $X_{t,s} = X_t-X_s$.

\section{Euler approximations and auxiliary results}

For $N\ge 1$ consider the following partition of the fixed interval $[0,T]:$
$\{0=\nu_0 < \nu_1< \dots < \nu_N = T,\ \delta=T/N\},\
\nu_k=k\delta.$

The Euler approximation for equation  \eqref{main} is defined recursively as
\begin{equation*}
\begin{gathered}
X_{\nu_{k+1}}^{\delta} =X_{\nu_k}^{\delta} +a(\nu_k,X_{\nu_k}^{\delta})\delta +
b(\nu_k,X_{\nu_k}^{\delta})\dw_k +c(\xd_{\nu_{k}})\db_k,
\end{gathered}
\end{equation*}
where $\dw_k = W_{\nu_{k+1},\nu_k}$, $\db_k = B^H_{\nu_{k+1},\nu_k}$.
The initial value of approximations is $X_{\nu_0}^{\delta}=X_0$.

Set $n^\delta_u=\max \{n:\nu_n \leq u \}$, $t_u^\delta = \nu_{n^\delta_u}$, and define continuous interpolation by
\begin{equation*} X_u^{\delta} =X_{t^\delta_u}^{\delta} +a(t^\delta_u,X_{t^\delta_u}^{\delta})(u-t^\delta_u)+
b(t^\delta_u,X_{t^\delta_u}^{\delta})W_{u,t^\delta_u}+c(X_{t^\delta_u}^{\delta})B^H_{u,t^\delta_u},\end{equation*}
or, in the integral form,
\begin{equation}\label{euler} X_u^{\delta} =X_0 +\int_0^u a(t^\delta_s,X_{t^\delta_s}^{\delta})ds+
\int_0^ub(t^\delta_s,X_{t^\delta_s}^{\delta})dW_s+\int_0^u c(X_{t^\delta_s}^{\delta})dB_s^H.
\end{equation}

The following lemma is a discrete analogue of the Gronwall inequality.
\begin{lemma}
\label{discgronwall}
If a non-negative sequence $\{x_n,n\ge 1\}$ satisfies
$$
x_{n+1} \le x_n (1+K\delta) + K\delta.
$$
Then
$$
x_{n} \le  (x_0 + 1) e^{K\delta n}.
$$
\end{lemma}
The following two lemmas are technical.
\begin{lemma}
For $s< \nu_n$, $n\ge 1$, one has
\begin{equation}\label{Ds formula}
D_s \xd_{\nu_n} = c(\xdc{s}) \prod_{k=n_s^{\delta}}^{n-1} \big(1+a'_x(\nu_k,\xd_{\nu_k})\delta + b'_x(\nu_k,\xd_{\nu_k})\dw_k + c'(\xd_{\nu_k})\db_k \big)
\end{equation}
(the product is set to $1$ when the upper limit is smaller than the lower).
\end{lemma}
\begin{proof}
Clearly, $D_s \xd_{\nu_n}=0$ if $\nu_n <s$. Now observe that $D_s \db_n = \ind{n=n^\delta_s} $. Hence, for $n=n^\delta_s$, we have
\begin{gather*}
D_s \xd_{\nu_n} = D_s \big(X_{\nu_{n-1}}^{\delta} +a(\nu_{n-1},X_{\nu_{n-1}}^{\delta})\delta +
b(\nu_{n-1},X_{\nu_{n-1}}^{\delta})\dw_{n-1}\\ +c(\xd_{\nu_{n-1}})\db_{n-1}\big)
= c(\xd_{\nu_{n-1}})D_s\db_{n-1} = c(\xd_{\nu_{n-1}}).
\end{gather*}
Further, for $n>n^\delta_s$ we can write
\begin{gather*}
D_s \xd_{\nu_n} = D_s \big(X_{\nu_{n-1}}^{\delta} +a(\nu_{n-1},X_{\nu_{n-1}}^{\delta})\delta\\
{} +
b(\nu_{n-1},X_{\nu_{n-1}}^{\delta})\dw_{n-1} +c(\xd_{\nu_{n-1}})\db_{n-1}\big)
=\big(1 +a'_x(\nu_{n-1},X_{\nu_{n-1}}^{\delta})\delta\\
{}+b'_x(\nu_{n-1},X_{\nu_{n-1}}^{\delta})\dw_{n-1} +c'(\xd_{\nu_{n-1}})\db_{n-1}\big)D_s \xd_{\nu_{n-1}},
\end{gather*}
and deduce \eqref{Ds formula} by induction.
\end{proof}

%
\begin{lemma}\label{lem:variation estimate}
For any $M>0$ it holds
$$
\ex{\exp\set{M\sum_{k=0}^{N-1}\left((\dw_k)^2+(\db_k)^2\right)}}<C_M
$$
for all $N$ large enough with $C_M$ independent of  $N$.
\end{lemma}
\begin{proof}
Using independence of $W$ and $B^H$, we then can write
\begin{gather*}
\ex{\exp\set{M\sum_{k=0}^{N-1}\left((\dw_k)^2+(\db_k)^2\right)}}\\
= \ex{\exp\set{M\sum_{k=0}^{N-1} (\dw_k)^2}}\ex{M\exp\set{\sum_{k=0}^{N-1} (\db_k)^2}}\\
\le  \prod_{k=0}^{N-1}\ex{\exp\set{M (\dw_k)^2}}\Big(\prod_{k=0}^{N-1}\ex{\exp\set{ M N (\db_k)^2}}\Big)^{1/N}\\
= C (1- 2 M\delta)^{-N/2} (1-2 M N \delta^{2H})^{-1/2}\\ = C (1-2 MT/N)^{-N/2} (1-2 M T^{2H}N^{1-2H})^{-1/2},
\end{gather*}
where the last equalities hold provided $2 M T/N<1$ and $2M T^{2H} N^{1-2H}<1$, which is true for all $N$ large enough.
Observing that
$$
C (1- 2 MT/N)^{-N/2} (1-2 M T^{2H}N^{1-2H})^{-1/2}\to C e^{MT}, \quad N\to\infty,
$$
we get the desired boundedness.
\end{proof}
Now we are ready to prove that the moments of Euler approximations as well as of their stochastic derivatives are uniformly bounded.
\begin{lemma}
For any $p>0$ one has
\begin{equation}
\label{Ds estimate}
\ex{\aabs{D_s \xd_{\nu_n}}^{p}}<C_p
\end{equation}
for all $s\in[0,T]$, $n\le N$,
with $C_p$ independent of $\delta$.
\end{lemma}
\begin{proof}
It is easy to see from \eqref{Ds formula} that the left-hand side of \eqref{Ds estimate} is finite. Therefore,
it suffices to establish boundedness only for $N$ large enough.

Introduce the following notation:
\begin{gather*}
a_k = a'_x(\nu_k,\xd_{\nu_k}), b_k = b'_x(\nu_k,\xd_{\nu_k}), c_k = c'(\xd_{\nu_k}),\\
\Theta_k = \aabs{1+ a_k \delta + b_k \dw_k + c_k \db_k} , \gamma_k = \abs{\dw_k} + \abs{\db_k},\\
d_k = a(\nu_k,\xd_{\nu_k})\delta+ b(\nu_k,\xd_{\nu_k})\dw_k, \Delta_k = d_k + c(\xd_{\nu_k})\db_k = \xd_{\nu_{k+1},\nu_k}.
\end{gather*}

Fix a small positive constant $\gamma$ (its value will be specified later to satisfy our needs). Put $A = \set{\forall k\ \gamma_k\le \gamma}$.
$$
\ex{\aabs{D_s \xd_{\nu_n}}^{p}} = \ex{\aabs{D_s \xd_{\nu_n}}^{p}\ind{A}} + \ex{\aabs{D_s \xd_{\nu_n}}^{p}\ind{\Omega\setminus A}}.
$$
\emph{Step 1}. First we estimate $\ex{\aabs{D_s \xd_{\nu_n}}^{p}\ind{\Omega\setminus A}}$. Write
\begin{gather*}
\ex{\aabs{D_s \xd_{\nu_n}}^{p}\ind{\Omega\setminus A}}= \sum_{B}
\ex{c(\xdc{s})^p\left(\prod_{k\notin B} \Theta_k\ind{\gamma_k\le \gamma}\prod_{k\in B}\Theta_k\ind{\gamma_k>\gamma} \right)^{p}},
\end{gather*}
where the outer sum is taken over all non-empty $B\subset \set{n_s^\delta,n_s^\delta+1,\dots,n-1}$. Observe that
$$
\aabs{a_k \delta + b_k \dw_k + c_k \db_k}\le K(\delta+\gamma_k),
$$
so this expression does not exceed 1 whenever $\delta<1/(2K)$ and $\gamma_k\le \gamma<1/(2K)$, and we can write
$$
\Theta_k \ind{\gamma_k\le \gamma} \le \exp\set{a_k \delta + b_k \dw_k + c_k \db_k}.
$$
For $\gamma_k>\gamma$ we estimate simply $\Theta_k<\exp\set{K(\delta+\gamma_k)}$,
therefore,
\begin{gather*}
\ex{\aabs{D_s \xd_{\nu_n}}^{p}\ind{\Omega\setminus A}}\\\le C_p\sum_{B}
\ex{\left(\exp\set{\sum_{k\notin B}(a_k \delta + b_k \dw_k + c_k \db_k)} \prod_{k\in B}\Theta_k\ind{\gamma_k>\gamma} \right)^{p}}\\
\le C_p\ex{\exp\set{p\sum_{k\notin B}(K \delta + b_k \dw_k + K\abs{ \db_k})}
\prod_{k\in B}e^{pK(\delta+\gamma_k)}\ind{\gamma_k>\gamma}}\\
\le C_p \Biggl(\ex{\exp\set{3p\sum_{k\notin B} b_k \dw_k }}\ex{\exp\set{3pK\sum_{k=0}^{N-1}\aabs{\db_k} }}\\
\times\ex{\prod_{k\in B}e^{3pK\gamma_k}\ind{\gamma_k>\gamma}}\Biggr)^{1/3}.
\end{gather*}
By the standard properties of the stochastic
integral with respect to $W$,
\begin{equation}\label{Dsestimate-1}
\begin{gathered}
\ex{\exp\set{3p\sum_{k\notin B} b_k \dw_k}} = \ex{\exp\set{\sum_{k\notin B} 9p^2 b_k^2 \delta/2 }}\\
\le \exp\set{5p^2 K^2 N\delta} = \exp\set{5p^2 K^2 T}.
\end{gathered}
\end{equation}
Now estimate, using the H\"older inequality,
\begin{equation}\label{Dsestimate-2}
\begin{gathered}
\ex{\exp\set{3p\sum_{k\notin B} c_k \db_k}}\le \ex{\exp\set{3pK\sum_{k=0}^{N-1}\aabs{\db_k} }}\\
 \le\left(\prod_{k=0}^{N-1} \ex{e^{3pK N \abs{\db_k}}}\right)^{1/N}\le \ex{e^{3pK N \abs{\db_0}}}\\\le Ce^{2p^2 K^2N^2 \delta^{2H}} = C e^{3p^2K^2 T^{2H} N^{2-2H}}.
\end{gathered}
\end{equation}
Further,
\begin{equation*}
\begin{gathered}
\ex{\prod_{k\in B}e^{3pK \gamma_k}\ind{\gamma_k>\gamma}}\le \ex{\prod_{k\in B}e^{3pK\gamma_k^2/\gamma}\ind{\gamma_k>\gamma}}\\
\le \left(\ex{\exp\set{\frac{6pK}\gamma\sum_{k=0}^{N-1} \gamma_k^2}}\ex{\prod_{j\in B}\ind{\gamma_j>\gamma}}\right)^{1/2}
\le C_{p,\gamma} \left(\ex{\prod_{j\in B}\ind{\gamma_j>\gamma}}\right)^{1/2},
\end{gathered}
\end{equation*}
where the last inequality hold for all $N$ large enough thanks to Lemma~\ref{lem:variation estimate}.
To estimate the last expectation, recall that $W$ and $B^H$ are independent and take first the expectation
with respect to $W$:
\begin{equation}
\label{Dsestimate-4}
\begin{gathered}
\ex{\prod_{j\in B}\ind{\gamma_j>\gamma}}\le  \ex{\prod_{j\in B}2\Phi\left(\big(\aabs{\db_j}-\gamma\big)\delta^{-1/2}\right)}\\
\le 2^{n(B)} \left(\prod_{j\in B}\ex{\Phi\left(\big(\aabs{\db_j}-\gamma\big)\delta^{-1/2}\right)^{n(B)}}\right)^{1/n(B)}\\ \le 2^{n(B)} \ex{\Phi\left(\big(\aabs{\db_0}-\gamma\big)\delta^{-1/2}\right)^{n(B)}},
\end{gathered}
\end{equation}
where $n(B)$ is the number of elements of $B$. We split the inner expectation into parts where $\abs{\db_0}\le\gamma/2$
and $\abs{\db_0}> \gamma/2$. For $\abs{\db_0}\le\gamma/2$ it holds $\Phi\left(\big(\aabs{\db_0}-\gamma\big)\delta^{-1/2}\right)
\le \Phi(-\gamma\delta^{-1/2}/2)^{n(B)}$, also we have $P(\abs{\db_0}> \gamma/2) \le 2 \Phi(-\gamma\delta^{-H}/2)$, hence
\begin{gather*}
\ex{\Phi\big((\abs{\db_0}-\gamma)\delta^{-1/2}\big)^{n(B)}}\le \Phi(-\gamma\delta^{-1/2}/2)^{n(B)} + 2 \Phi(-\gamma\delta^{-H}/2)\\
\le e^{-\gamma^2\delta^{-1}n(B)/8 } + e^{-\gamma^2\delta^{-2H}/2} \le  e^{-C_\gamma n(B) N } + e^{-C_\gamma N^{2H}}.
\end{gather*}
Plugging this into \eqref{Dsestimate-4}, we get
\begin{gather*}
\ex{\prod_{j\in B}\ind{\gamma_j>\gamma}}\le  \ex{\prod_{j\in B}2\Phi\big((\abs{\db_j}-\gamma)\delta^{-1/2}\big)}\\
\le e^{C(1-C_\gamma N) n(B) } + e^{C n(B) -C_\gamma N^{2H}}
\end{gather*}
and combining this with \eqref{Dsestimate-1} and \eqref{Dsestimate-2},
we arrive to
\begin{gather*}
\ex{\aabs{D_s \xd_{\nu_n}}^{p}\ind{\Omega\setminus A}}\\\le C_{p,\gamma} e^{C_p N^{2-2H}}\left( \sum_{B}  e^{C(1-C_\gamma N) n(B)} +
 e^{-C_\gamma N^{2H}} \sum_B e^{Cn(B)}\right).
\end{gather*}
For $N$ large enough it holds $C(C_\gamma N -1)\ge C_\gamma N$ (naturally, with different
constants $C_\gamma$ in the left-hand and in the right-hand sides), so  the first sum is bounded from above by
\begin{gather*}
\sum_{B} e^{-C_\gamma N n(B)} = \left(1+e^{-C_\gamma N}\right)^{n-n^\delta_s} - 1  \\
\le \left(1+e^{-C_\gamma N}\right)^{N} - 1 = \exp\set{N\log\left(1+e^{-C_\gamma N}\right)}-1\\
\le \exp\set{C N e^{-C_\gamma N}}-1\le \exp\set{C_\gamma e^{-C_\gamma N}}-1\le C_\gamma e^{-C_\gamma N}.
\end{gather*}
Similarly, the second sum is bounded by
\begin{gather*}
e^{-C_\gamma N^{2H}} \sum_{B} e^{C n(B)} = e^{-C_\gamma N^{2H}}\left(1+e^{C}\right)^{n-n_s^\delta} \le C_\gamma  e^{CN-C_\gamma N^{2H}}.
\end{gather*}
Since $H\in(1/2,1)$, this implies
\begin{gather*}
\ex{\aabs{D_s \xd_{\nu_n}}^{p}\ind{\Omega\setminus A}}\le C_{p,\gamma} e^{C_p N^{2-2H}}\left( e^{-C_\gamma N}  +
e^{CN-C_\gamma N^{2H}} \right)\\
\le C_{p,\gamma} \left(e^{-C_{p,\gamma}N}  + e^{-C_{p,\gamma} N^{2-4H}}\right)
\end{gather*}
for all $N$ large enough. This expression vanishes as $N\to\infty$, hence we get the boundedness of $\ex{\aabs{D_s \xd_{\nu_n}}^{p}\ind{\Omega\setminus A}}$.

\emph{Step 2}. Now we turn to $\ex{\aabs{D_s \xd_{\nu_n}}^{p}\ind{A}}$. If we take $\gamma<K^{-3}/3$ and $\delta<K^{-3}/3$, then
$\aabs{\Delta_k}<2 K^{-2}/3$ on $A$ and $\aabs{c(\xd_{\nu_{k+1}}) - c(\xd_{\nu_k})}< 2K^{-1}/3$. But $c(\xd_{\nu_k})>K^{-1}$,
so $c(\xd_{\nu_{k+1}})/c(\xd_{\nu_k})\in (1/3,5/3)$, which allows us to write by the Taylor formula
\begin{equation}\label{logC}
\begin{gathered}
\log \frac{c(\xd_{\nu_{k+1}})}{ c(\xd_{\nu_k})} = \frac{c'(\xd_{\nu_k})}{c(\xd_{\nu_k})}\Delta_k + R'_k
=  c_k\db_k + \frac{c_k}{c(\xd_{\nu_k})}d_k + R'_k,
\end{gathered}
\end{equation}where $\abs{R'_k}\le C \Delta_k^2$.
Similarly, on $A$
\begin{gather*}
\log \Theta_k = \log \left(1+ a_k\delta + b_k \dw_k + c_k \db_k\right) = a_k\delta + b_k \dw_k + c_k \db_k + R''_k
\end{gather*}
with $\abs{R''_k}\le C\Delta_k^2$. Plugging into this formula the expression for $c_k\db_k$ from \eqref{logC},
we get
\begin{gather*}
\Theta_k = \frac{c(\xd_{\nu_{k+1}})}{c(\xd_{\nu_k})}\exp\set{\alpha_k \delta + \beta_k \dw_k + R_k},
\end{gather*}
where $R_k = R''_k - R'_k$, $\alpha_k = a_k - {c_k a(\xd_{\nu_k})}/{c(\xd_{\nu_k})}$,
$\beta_k = b_k - {c_k b(\xd_{\nu_k})}/{c(\xd_{\nu_k})}$.
Now we can estimate
\begin{gather*}
\ex{\aabs{D_s \xd_{\nu_n}}^{p}\ind{ A}} = \ex{c(\xdc{s})^p\prod_{k=n_s^{\delta}}^{n-1}\Theta_k^p\ind{A}}\\
= \ex{c(\xd_{\nu_{n-1}})^p\exp\set{p\sum_{k=n_s^{\delta}}^{n-1} \left(\alpha_k \delta + \beta_k \dw_k+R_k\right)}}\\
\le C_p \left(\ex{\exp\set{2p\sum_{k=n_s^{\delta}}^{n-1}\beta_k \dw_k}}\ex{\exp\set{2p\sum_{k=0}^{N-1}\abs{R_k}}}\right)^{1/2}\\
\le C_p \left(\ex{\exp\set{2p^2\sum_{k=n_s^{\delta}}^{n-1}\beta^2_k \delta}}\ex{\exp\set{C_p\sum_{k=0}^{N-1}\Delta_k^2}}\right)^{1/2}\\
\le C_p \ex{\exp\set{C_p\sum_{k=0}^{N-1}\gamma_k^2}}\le C_p,
\end{gather*}
where the last holds for all $N$ large enough due to Lemma \ref{lem:variation estimate}. This completes the proof.
\end{proof}

\begin{lemma}
For any $p>0$ one has
\begin{equation}
\label{Xd estimate}
\ex{\aabs{\xd_{t}}^{p}}\le C_p
\end{equation}
for all $t\in[0,T]$.

Moreover,
\begin{equation}\label{Xdcont}
\ex{\abs{\xd_{t} -\xdc{t}}^p}\le C_p\delta^{p/2}
\end{equation}
for any $t\in[0,T]$.
\end{lemma}
\begin{proof}
It is enough to prove this for $p=2m$, $m\in\mathbb N$.
We first prove \eqref{Xd estimate} for $t=\nu_n$, using an induction by $m$.

Start with $m=1$.

Denote $a_n = a(\nu_n,\xd_{\nu_n}), b_n =b(\nu_n,\xd_{\nu_n}), c_n = c(\xd_{\nu_n})$ and
write for $\delta\in(0,1/2)$ by Jensen's inequality
\begin{gather*}
\ex{\big(\xd_{\nu_{n+1}}\big)^{2}}\le \ex{\big(\xd_{\nu_{n}} + b_n\dw_n+ c_n\db_n\big)^{2}}(1-\delta)^{-1} +
2\ex{\big(a_n\delta\big)^{2}}\delta^{-1}\\
\le \Big(\ex{\big(\xd_{\nu_{n}}\big)^2} + \ex{(b_n\dw_n)^2}+ \ex{(c_n\db_n)^{2}}\\
+ 2\ex{\xd_{\nu_{n}}b_n\dw_n}+2\ex{b_nc_n\,\dw_n\, \db_n}+2\ex{\xd_{\nu_{n}}\db_n} \Big)e^{2 \delta} + C \delta\\
\le \Big(\ex{\big(\xd_{\nu_{n}}\big)^2} + C\delta + C\delta^{2H}+2\ex{\xd_{\nu_{n}}\db_n} \Big)e^{2 \delta} + C \delta\\
\le \ex{\big(\xd_{\nu_{n}}\big)^2}e^{2\delta} + \ex{\xd_{\nu_{n}}\db_n} e^{2 \delta} + C\delta.
\end{gather*}
By \eqref{iBp} and \eqref{Ds estimate}, we can write
\begin{gather*}
\ex{\xd_{\nu_{n}}\db_n} = \alpha_H\int_0^{\nu_n} \int_{\nu_n}^{\nu_{n+1}} \ex{D_s\xd_{\nu_n}}(t-s)^{2H-2} dt\, ds \\
\le C  \int_{\nu_n}^{\nu_{n+1}} (t-\nu_n)^{2H-1} dt\le C\delta.
\end{gather*}
Then by Lemma~\ref{discgronwall}
$$
\ex{\big(\xd_{\nu_{n}}\big)^{2}}\le X_0^2 e^{C\delta n} \le Ce^{C\delta N} \le C,
$$
as required.

Now let $m\ge 2$ and for $l\le m$
\begin{equation*}
\ex{\big(\xd_{t}\big)^{2l}}\le C_{2l}.
\end{equation*}
In the further estimates constants may depend on $m$, but not on $n$.

Observe that by the Jensen inequality for $\delta<1$
$$
(a + b)^{2m} \le (1-\delta)^{1-2m} a^{2m} + \delta^{1-2m} b^{2m},
$$
whence
\begin{equation}
\label{jensen}
(a + b)^{2m} \le  a^{2m} (1+C_m\delta) + C_mb^{2m}\delta^{1-2m},
\end{equation}
therefore
\begin{gather*}
\ex{\big(\xd_{\nu_{n+1}}\big)^{2m}}\le \ex{\big(\xd_{\nu_{n}} + b_n\dw_n+ c_n\db_n\big)^{2m}}(1+C\delta) \\{}+
\ex{\big(a_n\delta\big)^{2m}}\delta^{1-2m}
\le \ex{\big(\xd_{\nu_{n}} + b_n\dw_n+ c_n\db_n\big)^{2m}}(1+C\delta) + C\delta.
\end{gather*}
Expand the power in the first term and consider a generic term of this expansion (without a coefficient):
\begin{gather*}
\ex{\big(\xd_{\nu_n}\big)^{2m-i-k}(b_n\dw_n)^i (c_n \db_n)^k}\\ = \ex{\big(\xd_{\nu_n}\big)^{2m-i-k}b_n^ic_n^k \ex{(\dw_n)^i ( \db_n)^k\mid \F_{\nu_n}}}\\
= \ex{\big(\xd_{\nu_n}\big)^{2m-i-k}b_n^ic_n^k \ex{(\dw_n)^i} \ex{( \db_n)^k\mid \F_{\nu_n}}}\\ = \ex{\big(\xd_{\nu_n}\big)^{2m-i-k}b_n^ic_n^k ( \db_n)^k} \frac{i!}{2^{i/2} (i/2)!}\delta^{i/2} \ind{i\ \text{even}}.
\end{gather*}
Thus, we can write
\begin{gather*}
\ex{\big(\xd_{\nu_{n}} + b_n\dw_n+ c_n\db_n\big)^{2m}}\\ = \sum_{k=0}^{2m} \sum_{j=0}^m {2m \choose k,2j,2m-k-2j} \ex{\big(\xd_{\nu_n}\big)^{2m-2j-k}b_n^{2j} c_n^k ( \db_n)^k} \frac{(2j)!}{2^{j} j!}\delta^{j},
\end{gather*}
where $${a+b+c \choose a,b,c} = \frac{(a+b+c)!}{a!b!c!}
$$
is a trinomial coefficient.

For $k=0$, $j\ge 1$, the terms of this sum are bounded by $C\delta$ by the induction hypothesis and boundedness of $b_n,c_n$.

Further, for $k\ge 2$
\begin{gather*}
\ex{\big(\xd_{\nu_n}\big)^{2m-2j-k}b_n^{2j} c_n^k (\db_n)^k}\le C \ex{\abs{\xd_{\nu_n}}^{2m-2j-k} \abs{\db_n}^k}\\
\le C\ex{\big(\xd_{\nu_n}\big)^{2m}}^{1/\lambda}\ex{\abs{\db_n}^{k\eta}}^{1/\eta}\\\le C\Big(1+\ex{\big(\xd_{\nu_n}\big)^{2m}}\Big)\delta^{2H}\le C\Big(1+\ex{\big(\xd_{\nu_n}\big)^{2m}}\Big)\delta,
\end{gather*}
where $\lambda = 2m/(2m-2j-k)$, $\eta = \lambda/(\lambda-1)$; here we have used an obvious estimate
\begin{equation}
\label{obvineq}
\Big(\ex{\big(\xd_{\nu_n}\big)^{2m}}\Big)^{1/\lambda}\le 1+\ex{\big(\xd_{\nu_n}\big)}^{2m}.
\end{equation}
Now estimate the term with $k=1$, $j=0$, using formula \eqref{iBp}:
\begin{gather*}
\ex{\big(\xd_{\nu_n}\big)^{2m-1} c_n \db_n} = \int_0^{\nu_n} \int_{\nu_n}^{\nu_{n+1}}
\ex{D_s\left(\big(\xd_{\nu_n}\big)^{2m-1}c(\xd_{\nu_n})\right)}\psi(t,s) dt\, ds \\
= \int_0^{\nu_n} \int_{\nu_n}^{\nu_{n+1}}
\ex{\Big((2m-1)\big(\xd_{\nu_n}\big)^{2m-2}c(\xd_{\nu_n}) + \big(\xd_{\nu_n}\big)^{2m-1}c'(\xd_{\nu_n})\Big)D_s \xd_{\nu_n}}\\
\times H(2H-1)(t-s)^{2H-2} dt\, ds\\
\le C\Big(1+\ex{\big(\xd_{\nu_n}\big)^{2m}}\Big)\int_{\nu_n}^{\nu_{n+1}} (t-\nu_n)^{2H-1} dt\le  C\Big(1+\ex{\big(\xd_{\nu_n}\big)^{2m}}\Big)\delta.
\end{gather*}
Here, as above we have used the H\"older inequality, inequality \eqref{obvineq} and boundedness of moments of the stochastic derivative.
The terms with $k=1$, $j\ge 1$ are estimated similarly.

Collecting the estimates, we get
$$
\ex{\big(\xd_{\nu_{n+1}}\big)^{2m}}\le \ex{\big(\xd_{\nu_n}\big)^{2m}}(1+C\delta) + C\delta,
$$
so by Lemma~\ref{discgronwall} we get the desired boundedness.

Now write for $s\in[\nu_n,\nu_{n+1})$
\begin{gather*}
\ex{\abs{\xd_{s} -\xd_{\nu_n}}^p}\le C_p \Big(\ex{\abs{a_n(s-\nu_n)}^p} + \ex{\abs{b_n W_{s,\nu_n})}^p} + \ex{\abs{c_n B^H_{s,\nu_n}}^p}\Big)\\
\le C_p \big((s-\nu_n)^p + (s-\nu_n)^{p/2} + (s-\nu_n)^{pH}\big) \le C_p (s-\nu_n)^{p/2},
\end{gather*}
which gives \eqref{Xdcont} and together with \eqref{Xd estimate} for $t=\nu_n$ implies \eqref{Xd estimate} for all $t\in[0,T]$.
\end{proof}

\section{Rate of convergence}

Now we are ready to prove the main result about the mean-square rate of convergence of Euler approximations.
\begin{theorem}\label{thm-main}
Euler approximations \eqref{euler} for the solution of equation \eqref{main} satisfy
$$
\ex{\big(X_t-\xd_t\big)^2}\le C(\delta+\delta^{4H-2}).
$$
\end{theorem}
\begin{proof}
Define $\psi(x) = \int_0^x c(z)^{-1} dz$. It is clear that
$$
K^{-1}\abs{x-y}\le \abs{\psi(x)-\psi(y)}\le K\abs{x-y}.
$$

Write by the It\^o formula
$$
\psi(X_t) = \psi(X_0) + \int_0^t \Big(\alpha(s,X_s)ds + \beta(s,X_s) dW_s \Big) + B_t^H,
$$
where
$$
\alpha(s,x) = \frac{a(s,x)}{c(x)} - \frac{b(s,x)^2 c'(x)}{2 c(x)^2}, \quad \beta(s,x) = \frac{b(s,x)}{c(x)}.
$$
Similarly,
$$
\psi(\xd_t) = \psi(\xd_0) + \int_0^t \Big(\alpha(s,\xd_s)ds + \beta(s,\xd_s) dW_s \Big) + B_t^H - G_t^\delta,
$$
where
\begin{gather*}
G_t^\delta = \int_0^t \bigg[c^{-1}(\xd_s)\Big(\big(a(s,\xd_s)-a(\td_s,\xdc{s})\big)ds + \big(b(s,\xd_s)-b(\td_s,\xdc{s})\big)dW_s\Big)\\
+ c^{-1}(\xd_s)\big(c(\xd_s)-c(\xdc{s})\big)dB_s^H + \frac{c'(\xd_s)}{2c^2(\xd_s)}\big(b(s,\xd_s)-b(\td_s,\xdc{s})\big)^2ds\bigg]\\
=: \int_0^t \big(a^\delta_s\, ds + b^\delta_s\, dW_s + c^\delta_s dB_s^H + d_s^\delta\,ds\big).
\end{gather*}

Estimate
\begin{gather*}
\ex{\left(\int_0^t a^\delta_s\, ds\right)^2}\le C\int_0^t \ex{(a^\delta_s)^2} ds \\
\le C\int_0^t \ex{c(\xd_s)^{-2}\big(a(s,\xdc{s})-a(\td_s,\xdc{s})\big)^2 + \big(a(s,\xd_s)-a(s,\xdc{s})\big)^2}ds\\
\le C\bigg(\delta^{2H-1} + \int_0^t \ex{\big(\xd_{s,\td_s}\big)^2}ds\bigg) \le C\delta^{2H-1}.
\end{gather*}
Similarly,
\begin{gather*}
\ex{\left(\int_0^t d^\delta_s\, ds\right)^2}\le C\int_0^t \ex{(d^\delta_s)^2} ds
\le C\delta^{2H-1}.
\end{gather*}
and using the It\^o isometry,
\begin{gather*}
\ex{\left(\int_0^t b^\delta_s\, dW_s\right)^2} = \int_0^t \ex{(b^\delta_s)^2} ds \le C\delta^{2H-1}.
\end{gather*}

Further, by \eqref{forw-skor-rel}
\begin{gather*}
\int_0^t c^\delta_s\, dB^H_s = \int_0^t c^\delta_s\, \delta B^H_s + \iint_{[0,t]^2} D_s c^\delta_u \psi(s,u) ds\,du=:I'(\delta,t)+I''(\delta,t).
\end{gather*}
By the chain rule for the stochastic derivative,
\begin{gather*}
D_u c^\delta_s = D_u\Big(c^{-1}(\xd_s)\big(c(\xd_s)-c(\xdc{s})\big)\Big)  \\
=\frac{c'(\xd_s)}{c(\xd_s)^2}\big(c(\xd_s)-c(\xdc{s})\big)D_u\xd_s +
 c^{-1}(\xd_s)\big( c'(\xd_s)D_u\xd_s- c'(\xdc{s})D_u\xdc{s}\big)\\
= \frac{c'(\xd_s)}{c(\xd_s)^2}\big(c(\xd_s)-c(\xdc{s})\big)D_u\xd_s
+ c^{-1}(\xd_s)\big(c'(\xd_s)-c'(\xdc{s})\big)D_u\xd_s\\
+ c^{-1}(\xd_s)c'(\xdc{s})D_u\xd_{s,\td_{s}}\big) =: D_1(u,s)+D_2(u,s) +D_3(u,s).
\end{gather*}
Now
$$
\ex{D_1(u,s)^2}\le C\ex{\big(c(\xd_s)-c(\xdc{s})\big)^2 (D_u\xd_s)^2}\le C\ex{\big(\xd_s-\xdc{s}\big)^2}\le C\delta.
$$
Similarly,
$$
\ex{D_2(u,s)^2}\le  C\delta.
$$
Further, for $u\le \td_{s}$
\begin{multline*}
\xd_{s,\td_{s}} =D_u \xdc{s}\Big( a'_x(\td_s,\xdc{s})(s-\td_{s}) + b'_x(\td_{s},\xdc{s})W_{s,\td_{s}} +
c'(\xdc{s})B^H_{s,\td_{s}}\Big)
\end{multline*}
and
\begin{gather*}
\ex{D_3(u,s)^2}\le \ex{(D_u \xdc{s})^4}^{1/2}\\
\times\ex{\Big( a'_x(\td_s,\xdc{s})(s-\td_{s}) + b'_x(\td_{s},\xdc{s})W_{s,\td_{s}} + c'(\xdc{s})B^H_{s,\td_{s}}\Big)^4}^{1/2}
\le C\delta;
\end{gather*}
for $u\in[\td_s,s)$
$$
D_3(u,s) = c(\xdc{s})
$$
and $D_3(u,s) = 0$  for $u>s$. Thus
\begin{gather*}
\ex{I''(\delta,t)^2}\le C\iint_{[0,t]^2}\mathsf E\Big[D_1(u,s)^2+D_2(u,s)^2\\
{}+D_3(u,s)^2\ind{[0,\td_s]}(u)\Big]\psi(s,u)du\,ds
+ C\ex{\left(\int_0^t \int_{\td_{s}}^s \abs{c(\xdc{s})}\psi(s,u) du\,ds\right)^2} \\
\le C\left(\delta+ \left(\int_0^t \aabs{s-\td_s}^{2H-1} ds\right)^2 \right)\le C(\delta+\delta^{4H-2}).
\end{gather*}

By \eqref{skorintestim}
\begin{gather*}
\ex{I'(\delta,t)^2} \le  \int_{[0,t]^2} \ex{c^\delta_s\, c^\delta_u} \psi(s,u) ds\, du \\
+ \int_{[0,t]^4} \ex{D_u c^\delta_v\, D_s c^\delta_z} \psi(s,u) du\,dv\,ds\,dz.
\end{gather*}
The first term is estimated by $C\delta$ using that
$$
\ex{c^\delta_s\, c^\delta_u}\le \ex{(c_u^\delta)^2}+\ex{(c_s^\delta)^2}\le C\ex{\big(\xd_{s,\td_{s}}\big)^2}+\ex{\big(\xd_{u,\td_{u}}\big)^2}\le C\delta.
$$
In the second, we write $D_u c_v^\delta = D_1(u,v) + D_2(u,v) + D_3(u,v)$ and similarly for $D_s c_z^\delta$. For $D_1,D_2$ we use the Cauchy inequality and the above estimates to get a bound of $C\delta$. This also works for $D_3$ when $u\notin [\td_v,v)$ and $s\notin[\td_z,z)$.
Two remaining terms for $D_3$ are similar, take e.g.
\begin{gather*}
\int_0^t\int_0^t\int_0^t\int_{\td_v}^v  \ex{D_u c^\delta_v D_s c^\delta_z} \abs{s-u}^{2H-2}\abs{z-v}^{2H-2}du\,dv\,ds\,dz\\
\le C\delta^{2H-1}\int_0^t\int_0^t\int_0^t  \ex{ \aabs{D_s c^\delta_z}} \abs{z-v}^{2H-2} dv\,ds\,dz.
\end{gather*}
Again, if $s\notin[\td_z,z)$, the integral can be estimated by $C\delta^{1/2}$; for $s\in[\td_z,z)$ we get $\delta^{2H-1}$.
Ultimately,
$$
\ex{I'(\delta,t)^2}\le C(\delta+\delta^{4H-2})
$$
and adding this to the previous estimates, we get
$$
\ex{(G_t^\delta)^2}\le C(\delta+\delta^{4H-2}).
$$

So we can write
\begin{gather*}
\ex{\big(\psi(X_t)-\psi(\xd_t)\big)^2}\\\le C\left(\int_0^t\ex{\big(\al(s,X_s)-\al(s,\xd_s)\big)^2+\big(\beta(s,X_s)-\beta(s,\xd_s)\big)^2}ds
+ \ex{(G_t^\delta)^2}\right)\\
\le C\int_0^t \ex{\big(X_s-\xd_s\big)^2}ds + C(\delta+\delta^{4H-2})\\
\le C\int_0^t \ex{\big(\psi(X_s)-\psi(\xd_s)\big)^2}ds + C(\delta+\delta^{4H-2}).
\end{gather*}
By the Gronwall lemma,
$$
\ex{\big(\psi(X_t)-\psi(\xd_t)\big)^2}\le C(\delta+\delta^{4H-2}),
$$
hence
$$
\ex{\big(X_t-\xd_t\big)^2}\le C(\delta+\delta^{4H-2}),
$$
as required.
\end{proof}
\begin{remark}
The obtained estimate can also be written as $\big(\ex{\big(X_t-\xd_t\big)^2}\big)^{1/2}\le C (\delta^{1/2}\vee \delta^{2H-1})$, so the
mean-square rate of convergence for the mixed equation is the worst of the two rates for ``pure'' stochastic differential equation with Brownian motion, $C\delta^{1/2}$, and with fractional Brownian motion, $C\delta^{2H-1}$. As long as these estimates for pure equations are sharp (see \cite{k-p,nourdin}), we
get that in our case the estimate is sharp as well.

An interesting observation is that the value of the Hurst index where the rate of convergence changes is $H=3/4$. From \cite{cheredito}
it is known that the measure induced by the mixture of Brownian motion and independent fractional Brownian motion
is equivalent to the Wiener measure iff $H>3/4$. So in this case it is perhaps natural to expect that the rate of convergence
of Euler approximations is the same as for Brownian motion, and this is exactly what we see here.
\end{remark}

\providecommand{\bysame}{\leavevmode\hbox to3em{\hrulefill}\thinspace}
\providecommand{\MR}{\relax\ifhmode\unskip\space\fi MR }
\providecommand{\MRhref}[2]{%
  \href{http://www.ams.org/mathscinet-getitem?mr=#1}{#2}
}
\providecommand{\href}[2]{#2}

\end{document}